\documentclass[12pt]{amsart}
\usepackage[all]{xy}
\usepackage{amsmath,amsthm,amsfonts,amssymb,latexsym,graphics}

\newtheorem{theorem}{Theorem}
\newtheorem{lemma}[theorem]{Lemma}
\newtheorem{proposition}[theorem]{Proposition}

\newtheorem{remark}[theorem]{Remark}
\theoremstyle{definition}

\newcommand{\obdry}{\partial \Omega}

\newcommand{\bte}{B(x_{ij},2\epsilon)}
\newcommand{\boe}{B(x_{ij},\epsilon)}
\newcommand{\eit}{e^{\frac{2\varphi_i}{h}}}
\newcommand{\annu}{\mathcal{A}_i}

\newcommand{\real}{\text{Re\,}}

\newcommand{\Cbb}{\mathbb{C}}
\newcommand{\Rbb}{\mathbb{R}}
\newcommand{\leftexp}[2]{{\vphantom{#2}}^{#1}{#2}}

\numberwithin{theorem}{section}

\SelectTips{cm}{}

\begin{document}

\title[Exponential Lower Bounds]{Exponential Lower Bounds for Quasimodes of Semiclassical Schr\"{o}dinger
Operators}
\author{Michael VanValkenburgh}
\address{Department of Mathematics \\ UCLA \\ Los Angeles, CA 90095-1555, USA}
\email{mvanvalk@ucla.edu}

\begin{abstract}
We prove quantitative unique continuation results for the
semiclassical Schr-\"{o}dinger operator on smooth, compact
domains. These take the form of exponentially decreasing (in $h$)
local $L^{2}$ lower bounds for exponentially precise quasimodes.
We also show that these lower bounds are sharp in $h$, and that,
moreover, the hypothesized quasimode accuracy is also sharp.
\end{abstract}

\maketitle

\section{Introduction}

In this paper we establish quantitative unique continuation
results for the semiclassical Schr\"{o}dinger operator on smooth,
compact domains. We consider a smooth, open, bounded, and
connected domain $\Omega\subset\Rbb^{n}$, and we let $G=(g^{ij})\in
C^{\infty}(\overline{\Omega})^{n^{2}}$ be a positive definite
symmetric matrix with real entries. Then, with $\Delta$ denoting
the ``Laplacian'' associated to this matrix,
\begin{equation*}
    \Delta=\sum_{i,j}\partial_{x^{i}}\,g^{ij}(x)\,\,\partial_{x^{j}},
\end{equation*}
and with $V\in C^{\infty}(\overline{\Omega},\Rbb)$ as our
potential, we take as our Schr\"{o}dinger operator
\begin{equation*}
    P(h):= -h^{2}\Delta +V.
\end{equation*}
For simplicity, we will only consider the Dirichlet realization of
$P$; that is, we will only allow $P$ to act on the domain
\begin{equation*}
    \mathcal{D}(P):=H^{2}(\Omega)\cap H^{1}_{0}(\Omega)
\end{equation*}
corresponding to ``zero boundary conditions''. Our unique
continuation results will take the form of local $L^{2}$ lower
bounds for certain quasimodes of this operator. For a uniformly
bounded spectral parameter
\begin{equation*}
    E(h)\in [a,b],\quad\text{for some  }-\infty< \min V\leq a \leq b<\infty,
\end{equation*}
and for some $\beta>0$ and $h_{0}>0$, we consider
$(\beta,h_{0})$-exponentially precise quasimodes of $P(h)$:
\begin{equation*}
    u(\cdot\,;h)\in \mathcal{D}(P)\quad\text{such that}\quad
    \begin{cases}
    ||u||_{L^{2}(\Omega)}=1, &\text{ and}\\
    ||(P(h)-E(h))u||_{L^{2}(\Omega)}=\mathcal{O}(e^{-\frac{\beta}{h}})
    \end{cases}
\end{equation*}
for all $h\in(0,h_{0})$. Throughout this paper we allow
$\beta=\infty$, which corresponds to exact eigenfunctions.

\vspace{10pt}

The following theorems are our main results:

\vspace{10pt}

\begin{theorem}\label{T:One}
    Let $\omega$ be an open subset of $\Omega$.
    Then there exist constants $C>0$, $\alpha>0$, $h_0 >0$, and $\beta_{0}>0$ such that
    \begin{equation*}
        C e^{-\frac{\alpha}{h}} \leq ||u(\cdot\,;h)||_{L^2(\omega)}
    \end{equation*}
    for all $(\beta,h_{0})$-exponentially
    precise quasimodes $u$ with $\beta>\beta_{0}$.
\end{theorem}

\vspace{10pt}

\begin{theorem}\label{T:Two}
Let $\Gamma\subset\obdry$ be a connected component of the boundary
of $\Omega$. Then there exist constants $C>0$, $\alpha>0$, $h_0
>0$, and $\beta_{0}>0$ such that
    \begin{equation*}
        C e^{-\frac{\alpha}{h}} \leq ||hNu(\cdot\,;h)||_{L^2(\Gamma)}
    \end{equation*}
    for all $(\beta,h_{0})$-exponentially
    precise quasimodes $u$ with $\beta>\beta_{0}$.
\end{theorem}

\vspace{10pt}

Here, as in the rest of the paper, $n$ denotes the outward unit
normal,
$$\nabla^{i}=\sum_{j}(G^{\frac{1}{2}})^{ij}\partial_{x^{j}},\qquad \text{and }N=\sum_{i,j}n_{i}g^{ij}\partial_{x^{j}}.$$

\vspace{10pt}

We will give simple examples showing that these lower bounds are
sharp in $h$. Moreover, in both theorems the quasimode accuracy is
also sharp; that is, we will show that there are
$\mathcal{O}(e^{-\frac{\beta}{h}})$ quasimodes for which the
theorems do not hold, when $\beta>0$ is relatively small.

Despite the fact that the statements of our results are rather
simple and natural, they do not seem to be treated in the
literature, at least not in this context. We therefore believe
that a short, explicit proof could be useful. Results of this
type, stated as ``doubling properties'' of eigenfunctions of the
Laplacian on Riemannian manifolds, with or without boundary, have
been proven by Donnelly and Fefferman \cite{R:DF1}, \cite{R:DF2}.
Their Carleman estimate (or ``quantitative Aronszajn inequality'')
is different from the one used here
(Theorem~\ref{P:CarlemanEst2}), and it is valid for Lipschitz
metrics on smooth, closed manifolds, which allows them to use the
estimate after reflecting across the boundary. Jerison and Lebeau
further studied ``doubling properties'', but for \emph{sums} of
eigenfunctions of the Laplacian \cite{R:JL}. Moreover, we were
particularly inspired by Theorem 7.6 in the course notes of Evans and Zworski,
which gives exponential estimates from below for certain semiclassical Schr\"{o}dinger operators on
$\Rbb^{n}$ that are elliptic at infinity \cite{R:ZwSC}.

\vspace{10pt}

The basic tool in this paper is a boundary Carleman estimate,
which we now describe.

Let $\Omega_0$ and $G_{0}=(g^{ij}_{0})$ be temporary placeholders
for $\Omega$ and $G=(g^{ij})$. Then our semiclassical
Schr\"{o}dinger operator has principal symbol (in the sense of $h$-differential operators)
\begin{equation*}
    p(x,\xi)= \sum_{i,j}\xi_{i}\, g^{ij}_{0}(x) \, \xi_{j} + V,\quad (x,\xi)\in
    \overline{\Omega_{0}}\times\Rbb^{n},
\end{equation*}
and for $\varphi\in C^{\infty}(\overline{\Omega_{0}},\mathbb{R})$
we let
\begin{equation}\label{E:pphi}
    p_{\varphi}(x,\xi):=p(x,\xi+i\varphi^{\prime}_{x}),
\end{equation}
which is the leading semiclassical symbol of the conjugated operator
\begin{equation*}
    P_{\varphi}:= e^{\frac{\varphi}{h}}\circ P\circ
    e^{-\frac{\varphi}{h}}.
\end{equation*}
This operator is given explicitly by
\begin{equation*}
    P_{\varphi}=\sum_{i,j}\left(hD_{x^{i}}+i\varphi^{\prime}_{x^{i}}\right)\circ g_{0}^{ij}(x)\circ\left(hD_{x^{j}}+i\varphi^{\prime}_{x^{j}}\right)+V,\qquad D_{x}=\frac{1}{i}\partial_{x}.
\end{equation*}

Now suppose that $\varphi$ is a Carleman weight, meaning that
$\varphi\in C^{\infty}(\overline{\Omega_{0}},\mathbb{R})$ and that
\begin{equation}\label{E:CWdef}
    p_{\varphi}(x,\xi)=E(h) \Rightarrow
    \frac{1}{i}\{\overline{p_{\varphi}},p_{\varphi}\}(x,\xi)\geq c>0
\end{equation}
uniformly with respect to $h$, for some constant $c>0$. Here we are using the Poisson
bracket, given, for $f,g\in C^{\infty}$, by
\begin{equation*}
    \{f,g\}:=\sum_{j=1}^{n}\left(\frac{\partial f}{\partial \xi_{j}}\frac{\partial g}{\partial
    x^{j}}-\frac{\partial f}{\partial x^{j}}\frac{\partial g}{\partial
    \xi_{j}}\right).
\end{equation*}

With this set-up, we have the following boundary Carleman
estimate, which may be found as Proposition~3.2 of Burq's paper
\cite{R:BurqAJM}.

\begin{theorem}\label{P:CarlemanEst2}
Let $\Gamma$ be a union of connected components of
$\partial\Omega_0$, and let $\varphi$ be a Carleman weight on
$\overline{\Omega_{0}}$ such that $\nabla\varphi\neq 0$ on
$\Omega_0$ and such that $N\varphi\Big{|}_{\obdry_{0}}\neq 0$. If
$N\varphi\Big{|}_{\Gamma}<0$, then there exist constants $c>0$ and
$h_{1}>0$ such that
\begin{equation*}
    \begin{aligned}
    \int_{\Omega_0}|(P_{\varphi}(h)-E(h))f|^{2}
    + h\int_{\obdry_0\backslash\Gamma}\Big{\{}&|f|^{2}+|h\nabla f|^{2}\Big{\}}\\
    &\geq ch\int_{\Omega_0}\Big{\{}|f|^{2}+|h\nabla f|^{2}\Big{\}}
    \end{aligned}
\end{equation*}
for every $h\in(0,h_{1})$ and every $f\in
C^{\infty}(\overline{\Omega_{0}})$ with $f|_{\Gamma}\equiv 0$.
\end{theorem}

\vspace{10pt}

\begin{remark}\label{Re:BurqCO}
    \emph{
    Estimates of this type, and their application to unique
    continuation problems, have a long, distinguished history.
    H\"{o}rmander's classic text \cite{R:Ho63} contains a
    systematic treatment of such estimates in the boundary-less
    case. The estimates up to the boundary were originally proven
    by Lebeau and Robbiano in the case when $V\equiv 0$ and
    $E(h)\equiv 0$ \cite{R:LRChaleur}, and Burq later observed that their proof extends to more general
    operators, including semiclassical Schr\"{o}dinger operators of
    the type considered here \cite{R:BurqAJM}. In all cases, the
    proof uses a partition of unity to reduce
    to local results; in the presence of the boundary, a change of
    variables is then applied to locally straighten the boundary
    segment. This is possible because the induced error terms do
    not affect the estimate (but possibly taking a smaller $h_{1}>0$).}

    \emph{Moreover, Theorem~\ref{P:CarlemanEst2} may be generalized to a
    more useful form--the form used in this
    paper--due to the fact that it is at heart a local result. If
    the function $f$ vanishes in a neighborhood of some boundary
    component $\Gamma_{0}$, then the result still holds, even if
    the condition $N\varphi\Big{|}_{\Gamma_{0}}\neq
    0$ fails to hold.}
\end{remark}

A central problem in the use of Carleman estimates is the
construction of suitable Carleman weights, and a classical
technique is to convexify a function which has no critical points
(see, for example, \cite{R:Ho63}, p.205, and \cite{R:BurqActa}).
In the proof of Theorem~\ref{T:One} we put the critical points
inside the set $\omega$, then apply the Carleman estimate to the
complement of $\omega$. In the proof of Theorem~\ref{T:Two}, we
use two Carleman weights with a certain compatibility condition
that allows us to piece together two Carleman estimates; this
follows a method of Burq \cite{R:BurqActa}.

As pointed out by the referee, our method for constructing Carleman weights is similar to that used by Chae, Imanuvilov, and Kim in the context of control theory \cite{R:ChaeImanKim}. For a connected bounded domain $\Omega\subset\Rbb^{n}$ with boundary $\partial\Omega\in C^{2}$, they construct and then convexify a function $\psi\in C^{2}(\overline{\Omega})$ which vanishes on $\partial\Omega$ and has its critical points in a given fixed subdomain of $\Omega$. In our case, however, it is important that the normal derivatives of $\psi$ on connected components of $\partial\Omega$ have predetermined signs, and we do not need $\psi$ to vanish on $\partial\Omega$.

We prove Theorems~\ref{T:One} and \ref{T:Two} in
Sections~\ref{S:Useful} and \ref{S:BdryTrace}, respectively, where
we also give remarks on the sharpness of the estimates.

From now on we will omit ``(h)'' where the $h$-dependence is
obvious. And in stating estimates we sometimes find it convenient
to write $X\lesssim Y$ or $Y\gtrsim X$ whenever $X\leq CY$ for
some constant $C>0$, which could possibly depend on $n$, the
dimension of $\Omega$.

\vspace{10pt}

\section{A Local Lower Bound}\label{S:Useful}

In proving Theorem~\ref{T:One} we begin with a useful elliptic
estimate:

\begin{proposition}\label{L:Useful}
    Let $\chi$, $u$ $\in C^{\infty}(\overline{\Omega},\Cbb)$. Then
    \begin{equation*}
        h^{2}\int|\chi|^{2}|\nabla u|^{2} \lesssim
        \int_{\text{supp}\chi}(|(P-E(h))u|^{2}+|u|^{2})
        + h^{2}\int_{\partial\Omega}|\chi|^{2}|uNu|
    \end{equation*}
    for all $h$ small enough.
\end{proposition}

\begin{proof}

The first part of the proof is an integration by parts:
\begin{equation*}
    \begin{aligned}
    \int\Big((P-E(h))u\Big)\,\overline{u}|\chi|^{2}
    &= h^{2}\int\nabla u\cdot\nabla(\overline{u}|\chi|^{2}) -
    h^{2}\int_{\partial\Omega}\overline{u}|\chi|^{2}Nu\\
    &\qquad\qquad\qquad
    +\int(V-E(h))|\chi u|^{2}\\
    &= h^{2}\int|\chi|^{2}|\nabla u|^{2} + 2h^{2}\int\overline{u}\nabla u\cdot\real(\overline{\chi}\nabla\chi) \\
    &\qquad\qquad\qquad -
    h^{2}\int_{\partial\Omega}\overline{u}|\chi|^{2}Nu + \int (V-E(h))|\chi u|^{2}.
    \end{aligned}
\end{equation*}
Then, by elementary estimates,
\begin{equation*}
    \begin{aligned}
    h^{2}\int|\chi|^{2}|\nabla u|^{2}
    &\lesssim \int|\chi (P-E(h))u|^{2} + \int|\chi u|^{2}\\
    &\qquad\qquad\qquad
        + h^{2}\int_{\partial\Omega}|\chi|^{2}|uNu|
        + h^{2}\int|u\nabla\chi||\chi\nabla u|\\
    &\lesssim \int|\chi (P-E(h))u|^{2} +
    \int(|\chi|^{2}+|\nabla\chi|^{2})|u|^{2}\\
        &\qquad\qquad\qquad + h^{2}\int_{\partial\Omega}|\chi|^{2}|uNu| + h^{4}\int|\chi|^{2}|\nabla u|^{2}.
    \end{aligned}
\end{equation*}
We absorb the last term on the right side into the left side to
conclude the proof.
\end{proof}

\vspace{20pt}

We can apply Proposition~\ref{L:Useful} to quasimodes with ``zero
boundary conditions''. These necessarily belong to the domain of
(the Dirichlet realization of) $P$, which is
\begin{equation*}
    \mathcal{D}(P):=H^{2}(\Omega)\cap H^{1}_{0}(\Omega).
\end{equation*}
Hence the computations in the preceding proof are still justified.
We can equip this set of functions with semiclassical norms; for
instance, in the following lemma we control the semiclassical
Sobolev norm $H^{1}_{h}$, given by
\begin{equation*}
    ||u||_{H^{1}_{h}} := \left(\int(|u|^{2}+|h\nabla u|^{2}) \right)^{\frac{1}{2}}.
\end{equation*}

\begin{lemma}\label{Cor:Useful}
    Let $u\in H^{2}(\Omega)\cap H^{1}_{0}(\Omega)$ be such that
    \begin{equation*}
        ||(P-E(h))u||_{L^{2}(\Omega)}=\mathcal{O}(f(h))||u||_{L^{2}(\Omega)}
    \end{equation*}
    for some function $f\geq 0$. Also let $\omega$, $\tilde{\omega}$
    be open subsets of $\Omega$ such that
    $\tilde{\omega}\subset\subset\omega\subset\subset\Omega$. Then
    \begin{equation*}
        ||u||_{H^1_h(\tilde{\omega})}\lesssim
        ||u||_{L^2(\omega)}+\mathcal{O}(f(h))||u||_{L^{2}(\Omega)}.
    \end{equation*}
\end{lemma}

\begin{proof} Let $\chi\in C^{\infty}(\overline{\Omega})$ be such that $0\leq\chi\leq 1$,
$\chi\equiv 1$ on $\tilde{\omega}$, and such that
supp$\chi\subset\omega$. Then, by Lemma~\ref{L:Useful}, we have
\begin{equation*}
    \begin{aligned}
    ||u||_{H^1_h(\tilde{\omega})}^2 &= \int_{\tilde{\omega}}\Big{[}|u|^2 + |h\nabla
    u|^2\Big{]}\\
    &\lesssim \int_{\tilde{\omega}}|u|^2 +
    \int_{\text{supp}\chi}(|(P-E(h))u|^{2}+|u|^2)\\
    &\lesssim ||u||_{L^2(\omega)}^2+\mathcal{O}(f(h)^{2})||u||^{2}_{L^{2}(\Omega)}.
    \end{aligned}
\end{equation*}
\end{proof}

\vspace{10pt}

We now construct a Carleman weight in the standard way: by
``convexifying'' a function which has no critical points, an idea
that goes back at least to H\"{o}rmander's classic book
(\cite{R:Ho63}, p.205). Moreover, we will find a Carleman weight whose outward normal derivative is negative everywhere on $\partial\Omega$, so that in using the Carleman estimate, Theorem \ref{P:CarlemanEst2}, we may discard the boundary term.

It is convenient to start with a Morse
function--that is, a smooth real-valued function on $\Omega$
having no degenerate critical points. For this we may first take
\emph{any} $\psi_{00}\in C^{\infty}(\overline{\Omega})$ with
$N\psi_{00}|_{\partial\Omega}<0$. We then smoothly extend it to a
neighborhood of $\overline{\Omega}$, and approximate the extension
by a Morse function $\psi_{0}$ in the $C^{1}$ topology, so that
$N\psi_{0}|_{\partial\Omega}<0$. We can do this because, for any
smooth manifold $X$, Morse functions are dense in
$C^\infty(X,\mathbb{R})$ (see, for instance, \cite{R:GG}).
Moreover, we choose $\psi_{0}$ to be non-negative on $\overline{\Omega}$,
simply by adding a sufficiently large constant.

Now let $x_1$, . . . , $x_N$ be the (necessarily finitely many) critical
points of $\psi_0$ on $\overline{\Omega}$; we then know that they are away from $\partial\Omega$. Also let $\omega_{0}$ be an open subset of
$\Omega$ such that $\omega_0\subset\subset\Omega$.

\begin{lemma}
    There exists a diffeomorphism $\varkappa:\overline{\Omega}\rightarrow\overline{\Omega}$
    such that $\varkappa(x)=x$ near $\partial\Omega$ and
    such that $\varkappa(x_j)\in\omega_0$ $\forall j$.
\end{lemma}

\begin{proof} For each $x_j$ we take a simple smooth curve
$\gamma_j:[0,1]\mapsto\Omega$ such that $\gamma_j(0)=x_j$ and
$\gamma_j(1)\in\omega_0$. We may choose the curves such that the
$\gamma_j([0,1])$ are pairwise disjoint. Let $N_j$ be a
neighborhood of $\gamma_j([0,1])$ such that the $N_j$ are pairwise
disjoint.

We take a $C^\infty$ vector field $X_j$ such that
$X_j(\gamma_j(t))=\gamma_j^{\prime}(t)$ $\forall t\in [0,1]$ and
such that $X_j$ is zero outside of $N_j$.

Since $X_j$ is a compactly supported $C^{\infty}$ vector field, it
induces a flow which fixes $\overline{\Omega}\cap\complement N_j$
and which induces a diffeomorphism $\varkappa_j$ of
$\overline{\Omega}$, the time $1$ flow of $X_{j}$, taking $x_j$
into $\omega_0$. Then $\varkappa:=\varkappa_1 \circ \dotsb \circ
\varkappa_N$ is the desired diffeomorphism.
\end{proof}

\vspace{10pt}

Let $\psi:=\psi_0\circ\varkappa^{-1}$. Then $\psi$ has finitely
many critical points, all of which are contained in $\omega_0$,
and $N\psi\Big{|}_{\partial\Omega}<0$.

Finally, let $$\varphi:=e^{\gamma\psi},$$ where $\gamma>0$ is to
be determined.

\vspace{10pt}

\begin{proposition}\label{P:CWgamma}
    For $\gamma$ large enough, $\varphi$ is a Carleman weight on
    $\overline{\Omega}\backslash\omega_0$.
\end{proposition}

\begin{proof} We have $\varphi^\prime = \gamma e^{\gamma\psi}\psi^\prime$,
$\varphi^{\prime\prime}=e^{\gamma\psi}(\gamma^2\psi^\prime{}^t\psi^\prime+\gamma\psi^{\prime\prime})$,
and $p_\varphi=E(h)$ implies that ${}^t\xi G\varphi^\prime=0$ and
${}^t\xi G\xi + V ={}^t\varphi^\prime G\varphi^\prime+E(h)$. Hence
\begin{equation}\label{E:controlxi}
    p_\varphi=E(h)\quad \text{implies}\quad |\xi|\leq C\gamma
    e^{\gamma\psi}
\end{equation}
where the bound is independent of $h$, as $E(h)\in[a,b]$. We now compute
\begin{equation*}\label{E:PBras00b}
    \begin{aligned}
    \{\text{Re}p_\varphi,\text{Im}p_\varphi\}
        &= 4\leftexp{t}\xi G\varphi^{\prime\prime}G\xi + 4\leftexp{t}\varphi^\prime G\varphi^
            {\prime\prime}G\varphi^\prime +
            2G^\prime(\varphi^\prime,\varphi^\prime,G\varphi^\prime)\\
                &\qquad+ 4G^\prime(\varphi^\prime,\xi,G\xi) - 2
                G^\prime(\xi,\xi,G\varphi^\prime) + 2\{V,\leftexp{t}\varphi^\prime G\xi\}\\
        &= 4e^{\gamma\psi}\gamma\leftexp{t}\xi G\psi^{\prime\prime}G\xi+4e^{3\gamma\psi}(\gamma^4(\leftexp{t}\psi^\prime
            G\psi^\prime)^2 +\gamma^3\leftexp{t}\psi^\prime
            G\psi^{\prime\prime}G\psi^\prime)\\
                &\qquad+
                2e^{3\gamma\psi}\gamma^3 G^\prime(\psi^\prime,\psi^\prime,G\psi^\prime)
                    +4\gamma
                    e^{\gamma\psi}G^\prime(\xi,\psi^\prime,G\xi)\\
                    &\qquad -2\gamma e^{\gamma\psi}G^\prime(\xi,\xi,G\psi^\prime)
                    + 2\gamma e^{\gamma\psi}\{V,\leftexp{t}\psi^\prime G\xi\}\\
        &= 4e^{3\gamma\psi}(\gamma^4(\leftexp{t}\psi^\prime G\psi^\prime)^2 +
        O(\gamma^3))-2\gamma
        e^{\gamma\psi}\leftexp{t}V^{\prime}G\psi^{\prime}.
    \end{aligned}
\end{equation*}
where in the last line we have used (\ref{E:controlxi}). Since
$|\psi^\prime|
> 0$ in $\overline{\Omega}\backslash\omega_0$, and since $G$ is of course
positive definite, the $\gamma^{4}$ term dominates the
$\gamma^{3}$ term when $\gamma$ is sufficiently large. The term
with the potential is also dominated, since we have chosen
$\psi\geq 0$ for this very purpose. Hence $\varphi$ is a Carleman
weight on $\overline{\Omega}\backslash\omega_0$ for $\gamma>0$
large enough.
\end{proof}

\vspace{10pt}

\begin{proof} \emph{(of Theorem~\ref{T:One}.)} Let $\omega_0$, $\omega_1$, and $\omega_2$ be open subsets of
$\Omega$ such that
$\omega_0\subset\subset\omega_{1}\subset\subset\omega_{2}\subset\subset\omega$,
where the critical points of our chosen Carleman weight are in
$\omega_0$ as above.

Let $\chi\in C^{\infty}(\overline{\Omega})$ be such that
$0\leq\chi\leq 1$ and
\begin{equation*}
    \chi \equiv
    \begin{cases}
        0 &\text{near $\overline{\omega_1}$}\\
        1 &\text{near $\complement\omega_2$}.
    \end{cases}
\end{equation*}

Let
\begin{equation*}
    M_{1}:=\max_{\overline{\Omega}\backslash\omega_{1}}\varphi,\qquad
    M_{2}:=\max_{\overline{\omega_{2}}\backslash\omega_{1}}\varphi,\qquad\text{and}\qquad
    m:=\min_{\overline{\Omega}\backslash\omega_{1}}\varphi
\end{equation*}
and note that $M_{2}>m$ when $\varphi$ is our chosen Carleman
weight.

Using our chosen weight $\varphi$ we apply the boundary Carleman
estimate (Theorem~\ref{P:CarlemanEst2}) to
$f=e^{\frac{\varphi}{h}}\chi u$ on
$\overline{\Omega}\backslash\omega_0$, with $\Gamma =
\partial\Omega$ and where we use the fact that $\chi$ vanishes near $\partial\omega_{0}$ (see Remark~\ref{Re:BurqCO}):
\begin{equation*}\label{E:CarlemanApp34246gg}
    \begin{aligned}
    ch^{\frac{1}{2}}||e^{\frac{\varphi}{h}}\chi u||_{L^{2}(\Omega)}
    &\leq ||e^{\frac{\varphi}{h}}(P-E(h))\chi u||_{L^{2}(\Omega)}\\
    &= ||e^{\frac{\varphi}{h}}[P,\chi]u+e^{\frac{\varphi}{h}}\chi (P-E(h))u||_{L^{2}(\Omega)}\\
    &\lesssim he^{\frac{M_{2}}{h}}||u||_{H^{1}_{h}(\omega_2)}
        +e^{\frac{1}{h}(M_{1}-\beta)}\\
    &\lesssim he^{\frac{M_{2}}{h}}(||u||_{L^{2}(\omega)}+e^{-\frac{\beta}{h}}) + e^{\frac{1}{h}(M_{1}-\beta)}.
    \end{aligned}
\end{equation*}
We have used Lemma~\ref{Cor:Useful} in the last step.

Hence
\begin{equation*}
    \begin{aligned}
    e^{\frac{m}{h}}||\chi u||_{L^{2}(\Omega)}
    &\lesssim h^{\frac{1}{2}}e^{\frac{M_{2}-\beta}{h}}
    +h^{\frac{1}{2}}e^{\frac{M_{2}}{h}}||u||_{L^{2}(\omega)}
    +h^{-\frac{1}{2}}e^{\frac{1}{h}(M_{1}-\beta)}\\
    &\lesssim e^{\frac{M_{2}-\beta}{h}} + e^{\frac{M_{2}}{h}}||u||_{L^{2}(\omega)}
    +h^{-\frac{1}{2}}e^{\frac{1}{h}(M_{1}-\beta)}
    \end{aligned}
\end{equation*}
which gives, with $\alpha := M_{2}-m>0$,
\begin{equation*}
    \begin{aligned}
    1 &\lesssim ||\chi u||_{L^{2}(\Omega)} + ||u||_{L^{2}(\omega)}\\
    &\lesssim
    e^{\frac{\alpha-\beta}{h}} + e^{\frac{\alpha}{h}}||u||_{L^{2}(\omega)}+||u||_{L^{2}(\omega)}
    +h^{-\frac{1}{2}}e^{\frac{1}{h}(M_{1}+\alpha-M_{2}-\beta)}\\
    &\lesssim e^{\frac{\alpha-\beta}{h}} + e^{\frac{\alpha}{h}}||u||_{L^{2}(\omega)}
    +h^{-\frac{1}{2}}e^{\frac{1}{h}(M_{1}+\alpha-M_{2}-\beta)}.
    \end{aligned}
\end{equation*}
That is,
\begin{equation*}
    e^{-\frac{\alpha}{h}}-e^{-\frac{\beta}{h}}
    -h^{-\frac{1}{2}}e^{\frac{1}{h}(M_{1}-M_{2}-\beta)}
    \lesssim ||u||_{L^{2}(\omega)}.
\end{equation*}
proving the result for quasimodes of accuracy
$\mathcal{O}(e^{-\frac{\beta}{h}})$, with, say,
\begin{equation}\label{E:beta0}
    \beta>\alpha+\max_{\overline{\Omega}\backslash\omega_{1}}\varphi-\max_{\overline{\omega_{2}}\backslash\omega_{1}}\varphi
    =:\beta_{0} \,\,(\geq \alpha).
\end{equation}

\end{proof}

\vspace{10pt}

\begin{remark}\label{Re:QMint}\emph{
Theorem~\ref{T:One} is sharp in $h$, both in terms of the
quasimode accuracy and in terms of the lower bound. For the
former, we consider quasimodes in the case where Agmon estimates
are relevant; we construct these quasimodes by simply multiplying
an eigenfunction by a suitable cutoff function. To be precise, we
let $E\in\Rbb$ and let $V$ be a potential such that the compact
set (the classically allowed region)
\begin{equation*}
    K:=\{x\in\overline{\Omega};\, V(x)\leq E\}
\end{equation*}
is non-empty and is contained in $\Omega$; hence the classically forbidden region
\begin{equation*}
    \{x\in\overline{\Omega};\, V(x)>E\}\neq\emptyset
\end{equation*}
contains a neighborhood of $\partial\Omega$. We then let $\chi\in C^{\infty}_{0}(\Omega)$ be such that $\chi=1$ near
$K$. Then $\text{supp}\nabla\chi\subset\{x\in\Omega;\, V(x)>E\}$. We then consider a family of
Dirichlet eigenfunctions $u(\cdot\,;h)$ such that
\begin{equation*}
    \begin{cases}
    Pu=(E+\lambda(h))u,\\
    ||u||_{L^{2}(\Omega)}=1, &\text{and}\\
    \lambda(h)\rightarrow 0 &\text{as $h\rightarrow 0$.}
    \end{cases}
\end{equation*}
Then
\begin{equation*}
    [-h^{2}\Delta,\chi]u=-h^{2}(\Delta\chi)u-2h^{2}\nabla\chi\cdot\nabla u
\end{equation*}
and hence
\begin{equation*}
    \begin{aligned}
    ||(P-E-\lambda(h))(\chi u)||_{L^{2}(\Omega)}
    &\leq h^{2}||(\Delta\chi)u||_{L^{2}(\Omega)}+2h^{2}||\nabla\chi\cdot\nabla
    u||_{L^{2}(\Omega)}\\
    &\leq Ch^{2}(||u||_{L^{2}(\text{supp}\nabla\chi)}+||\nabla
    u||_{L^{2}(\text{supp}\nabla\chi)})\\
    &\leq Ch^{2}(e^{-\frac{\epsilon}{h}}+||\nabla u||_{L^{2}(\text{supp}\nabla\chi)})
    \end{aligned}
\end{equation*}
for some $\epsilon>0$, as given by Agmon estimates (see, for
example, the book of Dimassi and Sj\"{o}strand \cite{R:DimassiSj}
or that of Helffer \cite{R:Helff}). We now let $U$ be an open set
containing $\text{supp}\nabla\chi$ and such that $\overline{U}$ is
contained in the interior of the classically forbidden region.
Then Lemma~\ref{Cor:Useful}, combined with another Agmon estimate,
gives (possibly with a different $\epsilon>0$)
\begin{equation}\label{E:AgQM}
    \begin{aligned}
    ||(P-E-\lambda(h))(\chi u)||_{L^{2}(\Omega)}
    &\lesssim
    h^{2}e^{-\frac{\epsilon}{h}}+h||u||_{L^{2}(U)}\\
    &\lesssim h^{2}e^{-\frac{\epsilon}{h}}+he^{-\frac{\epsilon}{h}}
    \end{aligned}
\end{equation}
for all $h>0$ sufficiently small.
Moreover, these same Agmon estimates show that
\begin{equation*}
    ||\chi u||_{L^{2}(\Omega)}=1-\mathcal{O}(e^{-\frac{\delta}{h}})
\end{equation*}
for some $\delta>0$ and for all $h>0$ sufficiently small. Hence
$\chi u$ can be renormalized without affecting the estimate
(\ref{E:AgQM}), thus resulting in a normalized quasimode which
vanishes in an open set.}

\emph{In summary, for $e^{-\frac{\beta}{h}}$ quasimodes, with
$\beta$ sufficiently large, we have our lower bound. But there are
\emph{other} $e^{-\frac{\epsilon}{h}}$ quasimodes (with
$\epsilon>0$ related to the Agmon metric) which vanish identically
in an $h$-independent open subset whose closure is contained in
the interior of the classically forbidden region.}

\emph{Moreover, from this discussion of Agmon estimates, it is
clear that the lower bound in Theorem~\ref{T:One} is sharp in
$h$.}
\end{remark}

\vspace{10pt}

\begin{remark}\emph{
    It may be possible to extend the proof to smooth, compact, connected, and oriented Riemannian
    manifolds, with or without boundary. For example, if $M$ is such a
    manifold without boundary, we let $\omega_{0}\subset M$ be open. As before, let $\psi\in
    C^{\infty}(M)$ be a nonnegative Morse function such that $\nabla\psi\neq 0$ on
    $M\backslash\omega_{0}$. Then $\varphi:=e^{\gamma\psi}$, with
    $\gamma>>1$, is a Carleman weight on $M\backslash\omega_{0}$, so
    we can apply the interior Carleman estimate on
    $M\backslash\omega_0$ (see Remark~\ref{Re:BurqCO}).}

    \emph{We again have that the result is sharp in terms of $h$,
    as the following concrete example shows.
    On the sphere $S^2$, with usual spherical coordinates
    $$(x_{1},x_{2},x_{3})=(\sin\theta\cos\varphi,\sin\theta\sin\varphi,\cos\theta),$$ we
    consider the functions
    \begin{equation*}
        f_n(\theta,\varphi) = (\sin\theta)^{n}(\cos\varphi+
        i\sin\varphi)^n.
    \end{equation*}
    These are called ``zonal harmonics''.}

    \emph{Then, letting $\Delta$ denote the spherical Laplacian,
    \begin{equation*}
        \Delta = \frac{\partial^2}{\partial\theta^2}+
        \frac{\cos\theta}{\sin\theta}\frac{\partial}{\partial\theta}+
        \frac{1}{\sin^2\theta}\frac{\partial^2}{\partial\varphi^2},
    \end{equation*}
    we get
    \begin{equation*}
        -\Delta f_n = n(n+1)f_n.
    \end{equation*}}

    \emph{We must now study the norm of $f_n$:
    \begin{equation*}
        \begin{aligned}
        \int_{S^2}|f_n|^2
        &=4\pi\int_{0}^{1}(1-x^2)^n dx = 2\pi\int_{0}^{1}(1-t)^n
        t^{-\frac{1}{2}}dt\\
        &= 2\pi B(\frac{1}{2},n+1)\\
        &= 4^{n+1}\pi\frac{(n!)^2}{(2n+1)!}\\
        &\approx \frac{4\pi^{\frac{3}{2}}n^{\frac{1}{2}}}{2n+1},
        \end{aligned}
    \end{equation*}
    where
    \begin{equation*}
        B(\frac{1}{2},n+1)= 2\Big{(}\frac{2\times 4\times 6 ...(2n)}{1\times 3\times 5
        ...(2n+1)}\Big{)}
    \end{equation*}
    is a so-called beta function. The important
    point is that we get some power of $n$, which is inconsequential
    against an exponential factor.}

    \emph{
    Now for local estimates, we have
    \begin{equation*}
        \int_{\omega}|f_n|^2 dS =
        \iint_{\omega}|\sin\theta|^{2n+1} d\theta d\varphi.
    \end{equation*}
    If we are looking at a set $\omega$ where, say, $(0\leq)\sin\theta
    \leq e^{-1}$, we get
    \begin{equation*}
        \int_{\omega}|f_n|^2 dS \lesssim e^{-2n}.
    \end{equation*}}

    \emph{With $h^{-2}:=n(n+1)$ and letting $F_h$ denote the corresponding
    \emph{normalized} eigenfunction, we get that
    \begin{equation*}
        ||F_n||_{L^{2}(\omega)}\lesssim e^{-\frac{\alpha}{h}}
    \end{equation*}
    for some $\alpha>0$ for all $h>0$ small enough.}
\end{remark}

\vspace{10pt}

\section{A Lower Bound for Normal Derivatives}\label{S:BdryTrace}

We now turn to the proof of Theorem~\ref{T:Two}, where the main
ideas came from a careful reading of a paper of Burq
\cite{R:BurqActa}. Thus, following Burq, we will use ``compatible
Morse functions'', as constructed in the following proposition,
whose proof can be found in \cite{R:BurqActa}, Appendix A. We are
allowing $\Gamma=\partial\Omega$, in which case some of the
conditions are void. In any case, we take $\Gamma$ to be a
connected component of $\partial\Omega$.

\begin{proposition}
    There exist Morse functions $\psi_{1},\psi_{2}$ on a
    neighborhood of $\overline{\Omega}$ such that
    \begin{equation*}
        N\psi_{i}\Big{|}_{\obdry\backslash\Gamma}<0 \qquad
        \text{and} \qquad N\psi_{i}\Big{|}_{\Gamma}>0,\qquad i=1,2,
    \end{equation*}
    and such that, for $x\in\Omega$, we have
    \begin{equation}\label{E:compat}
        \{\nabla\psi_{i}(x)=0\}
        \Longrightarrow \{ \nabla\psi_{i+1}(x)\neq 0 \quad\text{and}\quad \psi_{i+1}(x)>\psi_{i}(x)
        \} \qquad (\psi_3 \equiv \psi_1).
    \end{equation}
\end{proposition}

\vspace{10pt}

We call (\ref{E:compat}) the ``compatibility condition''.

\vspace{10pt}

\begin{proof} \emph{(of Theorem~\ref{T:Two}.)} Let $\psi_1$ and $\psi_2$ be compatible Morse functions, which we may assume are
nonnegative. Let $\{x_{ij}\}$ be the (finitely many) critical
points of $\psi_i$ in $\overline{\Omega}$, and let $\epsilon>0$ be small enough so that
\newcounter{spacerule}
\begin{list}{(\roman{spacerule})}
            {\setlength{\leftmargin}{.5in}
             \setlength{\rightmargin}{.5in}
             \usecounter{spacerule}}

    \item the balls $\bte$ are all disjoint ($i$ and $j$ varying) and have closures contained in $\Omega$, and
    \item $\psi_{i+1}>\psi_i$ on $\bte$
    $\qquad(\psi_{3}\equiv\psi_{1})$.
\end{list}

\vspace{10pt}

Let $\chi_i\in C^{\infty}(\overline{\Omega})$, for $i=1,2$, be
such that $0\leq \chi_{i}\leq 1$ and such that
\begin{equation*}\label{E:cutoff0sdf}
    \chi_i =
    \begin{cases}
        0   &\text{near} \quad \overline{\bigcup_{j}\boe}\\
        1   &\text{near} \quad \bigcap_{j}\complement \bte\cap\overline{\Omega}.
    \end{cases}
\end{equation*}
Also, let
\begin{equation*}\label{E:Omegai}
    \Omega_{i}:=\overline{\Omega}\cap\bigcap_{j}\complement \boe,
\end{equation*}
so that $\nabla\psi_{i}\neq 0$ on $\Omega_{i}$.

We now let $\varphi_{i}:=e^{\gamma\psi_{i}}$, with $\gamma>0$
taken large enough so that $\varphi_i$ is a Carleman weight on
$\Omega_i$, which follows from Proposition~\ref{P:CWgamma}. Our
boundary Carleman estimate, Theorem~\ref{P:CarlemanEst2}, applied
to $f=\exp\left(\frac{\varphi_{i}}{h}\right)\chi_{i}u$ on
$\Omega_{i}$ then gives the upper bound
\begin{equation*}
    \begin{aligned}
    ch\int_{\Omega_{i}}\Big{\{}e^{\frac{2\varphi_{i}}{h}}|\chi_i u|^2
    &+e^{\frac{2\varphi_{i}}{h}}|\varphi_{i}^{\prime}\chi_{i}u+h\nabla(\chi_i u)|^2\Big{\}} \\
    &\leq \int_{\Omega_i}e^{\frac{2\varphi_{i}}{h}}|(P-E(h))(\chi_i
    u)|^2\\
    &\qquad\qquad\qquad + h\int_{\cup_{j}\partial\boe\cup\Gamma}\Big{\{}|e^{\frac{\varphi_{i}}{h}}\chi_i u|^2 + |hN(e^{\frac{\varphi_{i}}{h}}\chi_i
    u)|^2\Big{\}}\\
    &=\int_{\Omega_i}e^{\frac{2\varphi_{i}}{h}}|(P-E(h))(\chi_i
    u)|^2\\
    &\qquad\qquad\qquad + h\int_{\Gamma}e^{\frac{2\varphi_{i}}{h}}|hN(\chi_i u)|^2.
    \end{aligned}
\end{equation*}
Together with an elementary lower bound, this gives the estimate
\begin{equation*}
    \begin{aligned}
    h\int_{\Omega_{i}}\Big{\{}|\chi_i u|^2 &+ |h\nabla(\chi_i
    u)|^2\Big{\}}\eit\\
    &\lesssim \int_{\annu}|[P,\chi_i]u|^2 \eit +
    h\int_{\Gamma}|hNu|^2 \eit +e^{\frac{2(M_{i}-\beta)}{h}}
    \end{aligned}
\end{equation*}
where $M_{i}:=\max_{\Omega_{i}}\varphi_{i}$ and $\annu:=\cup_{j}(\bte\backslash\boe)$.

This implies that
\begin{equation*}
    \begin{aligned}
    \int_{\complement(\cup_{j}\bte)\cap\Omega}&\Big{\{}|u|^2 + |h\nabla
    u|^2\Big{\}}\eit\\
    &\lesssim h\int_{\annu}\Big{\{}|u|^2 + |h\nabla
    u|^2\Big{\}}\eit + \int_{\Gamma}|hNu|^2 \eit +
    h^{-1}e^{\frac{2(M_{i}-\beta)}{h}}.
    \end{aligned}
\end{equation*}

\vspace{10pt}

Adding the two estimates, for $i=1,2$, we get
\begin{equation*}
    \begin{aligned}
    \sum_{i=1}^{2}&\int_{\complement(\cup_{j}\bte)\cap\Omega}\Big{\{}|u|^2 + |h\nabla
    u|^2\Big{\}}\eit\\
    &\lesssim \sum_{i=1}^{2}\Big{[}h\int_{\annu}\Big{\{}|u|^2 + |h\nabla
    u|^2\Big{\}}\eit + \int_{\Gamma}|hNu|^2 \eit + h^{-1}e^{\frac{2(M_{i}-\beta)}{h}} \Big{]}\\
    &\lesssim \sum_{i=1}^{2}\Big{[}h\int_{\annu}\Big{\{}|u|^2 + |h\nabla
    u|^2\Big{\}}e^{\frac{2\varphi_{i+1}}{h}} + \int_{\Gamma}|hNu|^2
    \eit + h^{-1}e^{\frac{2(M_{i}-\beta)}{h}} \Big{]}
    \end{aligned}
\end{equation*}
with $\varphi_3\equiv\varphi_1$, where we have used that
$\psi_{i+1}>\psi_i$ on $\bte$, with $\psi_{3}\equiv\psi_{1}$ (see
(ii) above).

But $\mathcal{A}_1 \subset
\complement(\cup_{j}B(x_{2j},2\epsilon))\cap\Omega$ and
$\mathcal{A}_2 \subset
\complement(\cup_{j}B(x_{1j},2\epsilon))\cap\Omega$, so we can
absorb the ``$\mathcal{A}$'' terms. This gives
\begin{equation*}
    \begin{aligned}
    \int_{\complement(\cup_{j}B(x_{1j},2\epsilon))\cap\Omega}&\Big{\{}|u|^2 + |h\nabla
    u|^2\Big{\}}e^{\frac{2\varphi_1}{h}} +
    \int_{\complement(\cup_{j}B(x_{2j},2\epsilon))\cap\Omega}\Big{\{}|u|^2 + |h\nabla
    u|^2\Big{\}}e^{\frac{2\varphi_2}{h}}\\
    &\lesssim \int_{\Gamma}|hNu|^2 e^{\frac{2\varphi_1}{h}}
    + \int_{\Gamma}|hNu|^2
    e^{\frac{2\varphi_2}{h}}+h^{-1}e^{\frac{2(M_{1}-\beta)}{h}}+h^{-1}e^{\frac{2(M_{2}-\beta)}{h}}.
    \end{aligned}
\end{equation*}

\vspace{10pt}

We let
\begin{equation*}\label{E:M}
    \begin{aligned}
    M&:=\max(\max_{\Gamma}\varphi_1,\max_{\Gamma}\varphi_2),\\
    m&:=\min(\min_{\overline{\Omega}}\varphi_1,\min_{\overline{\Omega}}\varphi_2),\qquad\text{and}\\
    \tilde{M}&:=\max(M_{1},M_{2}),
    \end{aligned}
\end{equation*}
and we note that $M>m$, by the positivity of the (outward) normal
derivatives of the Carleman weights on $\Gamma$:
\begin{equation*}
    N\varphi_{i}\Big{|}_{\Gamma}>0.
\end{equation*}
We then have
\begin{equation*}
    e^{\frac{2m}{h}}\int_{\Omega}\Big{\{}|u|^2 + |h\nabla
    u|^2\Big{\}}\lesssim e^{\frac{2M}{h}}\int_{\Gamma}|hNu
    |^2+h^{-1}e^{\frac{2(\tilde{M}-\beta)}{h}}.
\end{equation*}
We may now simply omit the gradient term on the left side and take
$\beta$ such that $\beta>\tilde{M}-m=:\beta_{0}$. Hence there
exist $c_0>0$ and $h_0>0$ such that
\begin{equation*}
    c_0 e^{-\frac{(M-m)}{h}}\leq ||hNu||_{L^{2}(\Gamma)}
    \qquad \forall h\in(0,h_0)
\end{equation*}
hence proving the theorem.
\end{proof}

\vspace{10pt}

\begin{remark}\label{Re:bdryopt} \emph{
    Just as in the previous section, we may use Agmon estimates to
    show that the $h$-dependence in Theorem~\ref{T:Two} is sharp, both for the stated quasimode accuracy
    and for the lower bound. As in Remark~\ref{Re:QMint}, we consider the case when the classically allowed region
    \begin{equation*}
        \{x\in\overline{\Omega};\, V(x)\leq E\}
    \end{equation*}
    is a non-empty subset of the open set $\Omega$. Then a neighborhood of the boundary of $\Omega$ is contained in the classically forbidden region
    \begin{equation*}
        \{x\in\overline{\Omega};\, V(x)>E\}.
    \end{equation*}
    Then, precisely as in Remark~\ref{Re:QMint}, we can use a cutoff function to create
    exponentially precise quasimodes which vanish in an $h$-independent neighborhood of $\Gamma$.}

    \emph{To show that the lower bound in Theorem~\ref{T:Two}
    is sharp in terms of $h$, we recall a well-known argument for estimating normal derivatives
    of eigenfunctions; we learned this from papers of Burq \cite{R:BurqControl}
    and Hassell and Tao \cite{R:HassTao}, where the relevant estimates are called ``Rellich-type estimates''.
    For simplicity, we take $G$ to be the identity matrix.}

    \emph{
    \begin{lemma}
        Let $u(\cdot\,;h)$ be a Dirichlet eigenfunction of $P$.
        Then, for any differential operator $A$,
        \begin{equation}\label{E:RellichId}
        \int_{\Omega}u[P,A]u = h^{2}\int_{\partial\Omega}\frac{\partial u}{\partial
        n}Au.
        \end{equation}
    \end{lemma}
    \begin{proof}
        Let $E(h)$ be the eigenvalue corresponding to
        $u(\cdot\,;h)$. Then
        \begin{equation*}
            \begin{aligned}
            \int_{\Omega}u[P,A]u
            &=\int_{\Omega}\Big[ u(P-E(h))Au - Au(P-E(h))u \Big]\\
            &=h^{2}\int_{\Omega}\Big[ Au\Delta u-u\Delta Au \Big],
            \end{aligned}
        \end{equation*}
        and so, by Green's formula and the fact that $u$ vanishes
        on the boundary, we get the desired identity.
    \end{proof}
    }

    \emph{
    We now choose an operator $A$ so that
    $||\partial_{n}u||^{2}_{L^{2}(\partial\Omega)}$ is recoverable from
    (\ref{E:RellichId}). For this, we use so-called geodesic
    normal coordinates near $\partial\Omega$, that is, coordinates
    $(r,y)$ near $\partial\Omega$ such that $r$ is the distance to
    $\partial\Omega$. Then we choose $$A=\chi(r)\frac{\partial}{\partial
    r},$$ where $\chi\in C^{\infty}_{0}(\Rbb)$ and is such that, for some
    $\delta>0$,
    \begin{equation*}
        \chi=
        \begin{cases}
        1 &\text{for }0\leq r\leq \frac{\delta}{2}\\
        0 &\text{for }r\geq\delta.
        \end{cases}
    \end{equation*}
    We take $\delta>0$ so that the coordinates $(r,y)$ are smooth
    for $r\in[0,\delta]$. Then the right side of (\ref{E:RellichId}) is just
    \begin{equation*}
        \int_{\partial\Omega}\left|h\frac{\partial u}{\partial n}\right|^{2}.
    \end{equation*}
    As for the left side of (\ref{E:RellichId}), we simply consider
    \begin{equation*}
        \int_{\Omega}u[P,A]u=h^{2}\int_{\Omega}u[-\Delta,A]u+\int_{\Omega}u[V,A]u.
    \end{equation*}
    Letting $$N_{\delta}(\partial\Omega):=\{x\in\Omega;\, \text{dist}(x,\partial\Omega)\leq\delta\}$$ it is clear that $[V,A]$ is a smooth function, supported in $N_{\delta}(\partial\Omega)$, and that $[-\Delta,A]$ is a second-order differential operator with smooth coefficients supported in $N_{\delta}(\partial\Omega)$. Hence
    \begin{equation*}
        \begin{aligned}
        \int_{\partial\Omega}\left|h\frac{\partial u}{\partial n}\right|^{2}
        &=\left|\int_{\Omega}u[P,A]u\right|\\
        &\lesssim \int_{N_{\delta}(\partial\Omega)}\left(|u|^{2}+|h\nabla u|^{2}\right).
        \end{aligned}
    \end{equation*}
    }

    \emph{
    If $\delta>0$ is moreover small enough so that
    $N_{\delta}(\partial\Omega)$ is in the interior of the
    classically forbidden region, Agmon estimates, as in Remark~\ref{Re:QMint}, show that
    \begin{equation*}
        \int_{N_{\delta}(\partial\Omega)}\left(|u|^{2}+|h\nabla u|^{2}\right)
        \lesssim e^{-\frac{c}{h}}
    \end{equation*}
    for some $c>0$ and for all $h>0$ sufficiently small. So we
    finally arrive at the estimate
    \begin{equation*}
        \int_{\partial\Omega}\left| h\frac{\partial u}{\partial n}
        \right|^{2}\lesssim e^{-\frac{c}{h}}.
    \end{equation*}
    }
\end{remark}

\vspace{10pt}

\section*{Acknowledgement}

It is a pleasure to thank M. Hitrik for the many helpful discussions.

\end{document}